\documentclass[12pt,  reqno]{amsart}
\usepackage{a4wide}
\usepackage{dsfont}
\usepackage{amsmath, amsthm,amssymb}
\usepackage{extarrows}
\usepackage{epic}
\usepackage{indentfirst}
\usepackage{graphicx}
\usepackage{graphics}
\usepackage{pict2e}
\usepackage{tikz}
\usetikzlibrary{arrows}
\pgfarrowsdeclarecombine{triang}{triang}%
{triangle 90}{triangle 90}{triangle 90}{triangle 90}

\def\red{\textcolor{red}}
\newtheorem{thm}{Theorem}
\newtheorem{cor}{Corollary}
\newtheorem{lem}{Lemma}
\newtheorem{prop}{Proposition}

\theoremstyle{definition}
\newtheorem{defn}{Definition}

\theoremstyle{remark}
\newtheorem{rem}{Remark}

\numberwithin{equation}{section}

\def\la{\lambda}
\DeclareMathOperator\toht{(31\mhyphen 2)}

\DeclareMathOperator\thot{(2\mhyphen 13)}

\newcommand\set[1]{\left\{#1\right\}}

\newcommand\floor[1]{\left\lfloor#1\right\rfloor}

\def\Z{\mathbb{Z}}
\def\N{\mathbb{N}}
\def\S{\mathfrak{S}}

\def\And{\mathfrak{D}}
\def\and{\mathrm{d}}

\DeclareMathOperator\first{\mathsf{first}}
\DeclareMathOperator\last{\mathsf{last}}

\DeclareMathOperator\inv{inv}
\DeclareMathOperator\exc{exc}

\DeclareMathOperator\des{des}
\DeclareMathOperator\asc{asc}

\DeclareMathOperator\dd{dd}
\DeclareMathOperator\da{da}
\DeclareMathOperator\peak{pk}
\DeclareMathOperator\valley{val}
\mathchardef\mhyphen="2D
\DeclareMathOperator\les{(31\mhyphen 2)}
\DeclareMathOperator\less{(13\mhyphen 2)}
\DeclareMathOperator\res{(2\mhyphen 13)}
\DeclareMathOperator\ress{(2\mhyphen 31)}

\def\DD{\mathrm{\mathcal{G}}}

\def\312{\les}
\def\132{\less}
\def\231{\ress}
\def\Z{\mathbb{Z}}
\def\N{\mathbb{N}}
\def\S{\mathfrak{S}}

\usepackage{cite}
\usepackage{hyperref}

\title[On the $\gamma$-coefficients of Br\"and\'en's $(p,q)$-Eulerian polynomials]
{Br\"and\'en's $(p,q)$-Eulerian polynomials,  Andr\'e permutations and  continued fractions}
\author{Qiong Qiong Pan}
\address[Qiong Qiong Pan]{Univ Lyon,  Universit\'{e} Lyon 1, UMR 5208 du CNRS, Institut Camille Jordan,  43, blvd du 11 novembre 1918, F-69622 Villeurbanne Cedex, France}
\email{qpan@math.univ-lyon1.fr}
\author{Jiang Zeng}
\address[Jiang Zeng]{Univ Lyon, Universit\'{e} Lyon 1, UMR 5208 du CNRS,  Institut Camille Jordan,  43, blvd du 11 novembre 1918, F-69622 Villeurbanne Cedex, France}
\email{zeng@math.univ-lyon1.fr}
\date{\today}

\begin{document}
\maketitle
\begin{abstract} 
In 2008 Br\"and\'en proved  a $(p,q)$-analogue of the 
 $\gamma$-expansion formula for  
Eulerian polynomials and conjectured the divisibility of 
the $\gamma$-coefficient $\gamma_{n,k}(p,q)$  by
$(p+q)^k$.  As a follow-up, in 2012
Shin and Zeng  showed  that  the fraction $\gamma_{n,k}(p, q)/(p + q)^k$ is a polynomial in $\N[p,q]$.
The aim of this paper is to  give a combinatorial interpretation of  the  latter polynomial  in terms of
 Andr\'e permutations, a class of objects  first defined and studied by Foata, 
 Sch\"utzenberger and Strehl in the 1970s. It turns out that our result provides an 
 answer to  a recent  open problem  
 of  Han, which was the impetus of this paper.
\end{abstract}
\section{Introduction}
The Euler number $E_n$, namely the coefficient of $x^n/n!$ in the expansion of $\sec(x)+\tan(x)$,
is well studied and has many combinatorial interpretations and  different refinements; see~\cite{FS73, Vi81, St09,GSZ, SZ10, SZ12, JV14, MPP}.  It was  Andr\'e~\cite{An79} who  first  proved  that $E_n$ is the number of alternating permutations $a_1\ldots a_n$ of $12\ldots n$, i.e., $a_1>a_2<\cdots$.
Among the many remarkable identities for the Euler numbers there is the less known J-type continued fraction 
\begin{align}\label{ES-CF}
\sum_{n=0}^\infty E_{n+1}x^n=\cfrac{1}{1-x-\cfrac{x^2}{1-2x-\cfrac{3x^2}{1-3x-\cfrac{6x^2}{1-4x-\cfrac{10x^2}{1-\cdots}}}}}.
\end{align}
This formula does not appear in Flajolet's classic~\cite{Fl80} and its connection with 
the work of Stieltjes~\cite{St90} was unveiled only in 2018, see\cite{So18}.  More recently, 
 Han~\cite{Han19}  considered a  $q$-version of \eqref{ES-CF} and 
 asked for a combinatorial interpretation for the corresponding $q$-Euler numbers $E_n(q)$ (see \eqref{q-ES-D} below).  Motivated by  Han's question, we shall study the more general polynomials $D_n(p,q,t)$ defined by the following continued fraction, which is a $(p,q)$-analogue of (1.1):
\begin{align}\label{pq-shift}
\sum_{n=0}^\infty D_{n+1}(p,q,t)x^n
=\cfrac{1}{1-x-\cfrac{\binom{2}{2}_{p,q} t\,x^2}{1-[2]_{p,q} x-
\cfrac{\binom{3}{2}_{p,q}t\,x^2}{1-[3]_{p,q} x-\cfrac{\binom{4}{2}_{p,q}t\,x^2}{1-[4]_{p,q} x-\cfrac{\binom{5}{2}_{p,q} t\,x^2}{1-\cdots}}}}},
\end{align}
where  the $(p,q)$-analogue of $n$ is defined by

 $$
 [n]_{p,q}=\frac{p^n-q^n}{p-q}=\sum_{i+j=n-1}p^iq^j \qquad (n\in \N)
 $$
  and the $(p,q)$-analogue of the binomial coefficient $\binom{n}{k}$ is defined by
$$
\binom{n}{k}_{p,q}=\frac{[n]_{p,q}\ldots [n-k+1]_{p,q}}{[1]_{p,q}\ldots [k]_{p,q}}\qquad (0\leq k\leq n).
$$
Comparing  \eqref{ES-CF} and \eqref{pq-shift}  yields that 
$$
D_{n}(1,1,1)=E_{n} \qquad (n\geq 1).
$$
The $q$-Euler number $E_n(q)$  of Han~\cite{Han19} can be expressed  as
\begin{align}\label{q-ES-D}
E_n(q):=D_{n}(1,q,1)=D_{n}(q,1,1)\qquad  (n\geq 1).
\end{align} 

  The first few values of  $D_n(p,q,t)$  are $D_1(p,q,t)=D_2(p,q,t)=1$, and 
\begin{align}
D_3(p,q,t)&=1+t,\quad 
D_4(p,q,t)=1+(p+q+2)t,\label{eq:3-4}\\
D_5(p,q,t)&=1+  \left( (p+q)^2+2\,(p+q)+3 \right) t+\left( {p}^{2}+pq+{q}^{2}+1 \right) {t}^{2}.\nonumber
\end{align}

It turns out that the polynomials $D_n(p,q,t)$ are related to the $\gamma$-coefficients of 
Br\"and\'en's  $(p,q)$-analogue of Eulerian polynomials~\cite{Bra08}.  In this paper we shall interpret
$D_n(p,q,t)$ 
 in terms of  \emph{Andr\'e permutations}, which were introduced and studied by Foata, Sch\"utzenberger and  Strehl~\cite{FS73, FS74, FS76} in the 1970s.
There are three ingredients in our proof:    
the  connection of these polynomials with the $\gamma$-coefficients of  Br\"and\'en's $(p,q)$-analogue of Eulerian polynomials~\cite{Bra08}, Shin-Zeng's  continued fraction expansion  of  the $\gamma$-coefficients of generalized Eulerian polynomials~\cite{SZ12} and a new action on the permutations without double descents.

Recall that any polynomial $h(t)=\sum_{i=0}^n h_it^i$ satisfying   $h_i=h_{n-i}$ can be expressed uniquely in the form 
$\sum_{i=0}^{\lfloor{(n-1)/2\rfloor}} \gamma_{i} t^i
(1+t)^{n-2i}$. The coefficients $\gamma_{i}$  are called the $\gamma$-coefficients of $h(t)$. 
If the $\gamma$-coefficients  $\gamma_i$ are all nonnegative, then $h(t)$ is said 
to be $\gamma$-positive.  The unimodality of the sequence $(h_0, \ldots, h_n)$ 
 is a direct consequence of $\gamma$-positivity.   
 Let $\S_n$ be the set of permutations  of $[n]:=\set{1,\dots,n}$.
 For  a  permutation $\sigma:=\sigma_1\sigma_2\ldots \sigma_n$ of $[n]$,
the \emph{descent number}  $\des\sigma$ is the number of descent positions, i.e. $i<n$ such that  $\sigma_i>\sigma_{i+1}$, and the \emph{excedance number}  $\exc\sigma$ is the number of excedance
 positions, i.e. $i\in [n]$ such that  $\sigma_i>i$.
Thanks to  the work of MacMahon~\cite{Mac15} and Riordan~\cite{Ri51} we can define 
 the Eulerian polynomials $ A_n(t)$ by
 $$
 A_n(t)=\sum_{\sigma\in \S_n} t^{\des\sigma}=\sum_{\sigma\in \S_n} t^{\exc\sigma}.
 $$
 The following $\gamma$-decompositions for $A_n(t)$ 
 are well-known~\cite[Section~4]{FS73}.
\begin{thm} [Foata and Sch\"utzenberger]
\begin{align}
A_n(t)&=\sum_{k=0}^{\lfloor{n/2\rfloor}}\gamma_{n,k} t^k (1+t)^{n-1-2k}\label{eulerian:gamma}\\
&=\sum_{k=0}^{\lfloor{n/2\rfloor}}2^kd_{n,k} t^k (1+t)^{n-1-2k}, \label{eulerian:andre}
\end{align}
where  $\gamma_{n,k}=2^kd_{n,k}$ and   $d_{n,k}$ are positive integers satisfying  the recurrence
\begin{align}
 \and_{1,0}&=1\quad \textrm{and for $n\geq 2$, $k\geq 0$},\nonumber\\
 \and_{n,k}&=(k+1)\and_{n-1,k}+(n-2k)\and_{n-1,k-1}.
 \end{align} 
Moreover, the sum $\sum_k d_{n,k}$ is precisely the Euler number $E_n$,  see Figure~\ref{fig1}.
\end{thm}

In the last two decades although  various refinements of \eqref{eulerian:gamma} have been given in combinatorics and geometry (see \cite{ PRW,  Pe15, At18,SW20}), 
similar extension of  \eqref{eulerian:andre} does not seem to be known. 
In this paper we will provide two 
refinements  of  \eqref{eulerian:andre} (see \eqref{gamma-d} and \eqref{gamma-qd}).

\begin{figure}[t]\label{fig1}
  $$
  \begin{array}{c|ccccc}
\hbox{$n$}\backslash\hbox{$k$}&0&1&2&3\\
\hline
1& 1&\\
2& 1&\\
3& 1&2&&\\
4& 1&8&&\\
5& 1&22&16&\\
6& 1&52&136&\\
7& 1&114&720&272\\
\end{array}
\qquad\qquad\qquad
 \begin{array}{c|cccc|c}
 n\diagdown  k &0&1&2&3&E_n=\sum_k \and_{n,k}\\
 \hline
 1&1&&&&1\\
 2&1&&&&1\\
 3&1&1&&&2\\
 4&1&4&&&5\\
 5&1&11&4&&16\\
 6&1&26&34&&61\\
 7&1&57&180&34&272\\
 \end{array}
 $$
 \caption{The first values of $\gamma_{n,k}$ (left), $d_{n,k}$ and  $E_n$ for $0\leq 2k< n\leq 7$.}
 \end{figure}
 

\begin{defn}
For  a permutation $\sigma=\sigma_1\ldots \sigma_n$ of $[n]$ with  $\sigma_0=\sigma_{n+1}=0$,  the  entry $\sigma_i$ is 
\begin{itemize}
\item a \emph{peak}  if $\sigma_{i-1} < \sigma_{i}$ and $\sigma_{i} > \sigma_{i+1}$;
\item a \emph{valley}  if $\sigma_{i-1} > \sigma_{i}$ and $\sigma_{i} < \sigma_{i+1}$;
\item a \emph{double ascent}  if $\sigma_{i-1} < \sigma_{i}$ and $\sigma_{i} < \sigma_{i+1}$;
\item a \emph{double descent} if $\sigma_{i-1} > \sigma_{i}$ and $\sigma_{i} > \sigma_{i+1}$.
\end{itemize}
\end{defn}
Let $\peak\sigma$ (resp. $\valley \sigma$, $\da\sigma$, $\dd\sigma$) denote  the number of peaks (resp. valleys, double ascents, double descents) in~$\sigma$. 
Note that $\des\sigma=\valley\sigma+\dd\sigma$ and $\peak \sigma=\valley \sigma+1$.
Let $\mathcal{G}_{n,k}=\{\sigma\in \S_n: \valley\sigma=k, \dd\sigma=0\}$.  For example, the elements of $\mathcal{G}_{4,1}$ are 
$$
13\red{2}4, \quad 14\red{2}3,  \quad 2\red{1}34, \quad  23\red{1}4,   \quad 24\red{1}3, \quad 3\red{1}24,  \quad 34\red{1}2,  \quad 4\red{1}23.
$$

\begin{defn}
For a permutation $\sigma$ of $[n]$,  let  $\sigma_{[k]}$ be the \emph{subword} of $\sigma$ consisting of $1,\dots, k$ in the order they appear in $\sigma$. Then, 
the  permutation $\sigma$ is an  Andr\'e permutation 
if $\sigma_{[k]}$ has no double descents (and ends with an ascent) for all $1\leq k\leq n$.
\end{defn}
For example, the permutation $\sigma = 43512$ is not Andr\'e  
since the subword $4312$ of $\sigma$ contains a double descent $3$, while the permutation $\tau = 31245$ is Andr\'e.  Let $\And_n$ be the set of Andr\'e permutations of  $[n]$.
For instance, the set $\And_{4}$ consists of five permutations  
$1234, 1423, 3124, 3412$, and  $4123$.  
Let $\And_{n,k}$ be the set of Andr\'e permutations of $[n]$ with $k$ descents. For example, the elements of
$\And_{5,2}$ are 
$$
3\red{1}5\red{2}4,\; 4\red{1}5\red{2}3,\; 5\red{1}4\red{2}3,\,5\red{3}4\red{1}2.
$$

\begin{prop}[\cite{FS73,FS76}]
The coefficients $\gamma_{n,k}$  and $d_{n,k}$ 
equal  the cardinalities of $\mathcal{G}_{n,k}$ and  $\And_{n,k}$, respectively.
\end{prop}
For  $\sigma = \sigma_1 \dots \sigma_n \in \S_n$,
the statistic $\les \sigma$ 
 is the number of pairs $(i,j)$ such that $2\leq i<j\leq n$ and $ \sigma_{i-1}>\sigma_j>\sigma_i$.
 Similarly,
the statistic $\res\sigma $ 
is the number of pairs $(i,j)$ such that $1\leq i<j\leq n-1$ and $\sigma_{j+1}>\sigma_i>\sigma_j$. 
 In 2008 Br\"and\'en~\cite{Bra08} defined  a $(p,q)$-analogue of Eulerian polynomials and 
proved a $(p,q)$-analogue of \eqref{eulerian:gamma}. In this paper we shall use  the following 
 variant  of Br\"and\'en's $(p,q)$-Eulerian polynomials
in \cite{SZ12}
\begin{align}
A_n(p,q,t):=\sum_{\sigma\in \S_n} p^{\res \sigma} q^{\les \sigma}t^{\des\sigma}.
\end{align}
For $0\leq k\leq (n-1)/2$ define the $(p,q)$-analogue of  
$\gamma_{n,k}$ and $d_{n,k}$ in \eqref{eulerian:gamma} and \eqref{eulerian:andre} by
\begin{align} 
\gamma_{n,k}(p,q)&=\sum_{\sigma\in \mathcal{G}_{n,k}} p^{\res \sigma} q^{\les \sigma},\label{eq:combin}\\
d_{n,k}(p,q)&=\sum_{\sigma\in \And_{n,k}}p^{\res\sigma}q^{\312\sigma-k}.
\end{align} 
Our main results are the following two theorems.
\begin{thm} \label{thm:main1} We have 
\begin{align}
A_n(p,q,t)&=\sum_{k=0}^{\lfloor{(n-1)/2\rfloor}}\gamma_{n,k}(p,q)
t^{k}(1+t)^{n-1-2k}\label{gamma-pq}\\
&=\sum_{k=0}^{\lfloor{(n-1)/2\rfloor}}(p+q)^k\and_{n, k}(p,q)\label{gamma-d}
t^{k}(1+t)^{n-1-2k}.
\end{align}
\end{thm}
\begin{rem}
An equivalent $\gamma$-expansion of \eqref{gamma-pq} was  proved  by Br\"and\'en~\cite{Bra08} using the \emph{modified Foata-Stehl} action.  
The divisibility of  $\gamma_{n,k}(p,q)$ by $(p+q)^k$ was conjectured by 
Br\"and\'en~(\emph{op.cit.}) and proved by Shin and Zeng~\cite{SZ12} using the combinatorial theory of
continued fractions.
\end{rem}
\begin{thm}\label{thm:main2}  We have 
\begin{align}
D_{n}(p,q,t)&=\sum_{k=0}^{\lfloor (n-1)/2\rfloor} \and_{n,k}(p,q)t^k\label{eq:1}\\
&=\sum_{\sigma\in \And_{n}}p^{\res\sigma}q^{\312\sigma-\des\sigma}t^{\des\sigma}.\label{eq:2}
\end{align}
\end{thm}
\begin{rem}
It is not difficult to see that $\312\sigma\geq \des\sigma$ for any 
$\sigma\in \And_n$, see (iii) of Proposition~\ref{x-fact-andre}. 
\end{rem}
\begin{ex}
We enumerate the  Andr\'e permutations  in $\And_3$ and $\And_4$ with  their number of  patterns $\res$ and $\312$.
The valleys are in boldface.
\begin{center}
\begin{tabular}{|c|c|c|c|c|}
\hline
$\sigma\in \And_3$ & $\res\sigma$ & $\312\sigma$& $\des\sigma$\\\hline 
$123$ & $0$ & $0$&0\\\hline
$3\mathbf{1}2$ & $0$ & $1$&$1$\\\hline
\end{tabular}
\qquad 
\begin{tabular}{|c|c|c|c|}
\hline
$\sigma\in \And_4$ &  $\res\sigma$ &$\312\sigma$&$\des\sigma$\\\hline
$1234$ &  $0$ & $0$&0\\\hline
$14\mathbf{2}3$ &  $0$ & $1$&$1$\\\hline
$3\mathbf{1}24$ &  $1$ & $1$&$1$\\\hline
$34\mathbf{1}2$ &  $0$ & $1$&$1$\\\hline
$4\mathbf{1}23$ &  $0$ & $2$&$1$\\\hline
\end{tabular}
\end{center}
These are in  accordance with \eqref{eq:3-4}.
\end{ex}
\medskip

Combining Theorem~\ref{thm:main1} with Theorem~1 in \cite{SZ16}, which is \eqref{gamma-q} below, 
 we derive a $q$-analogue of 
\eqref{eulerian:gamma} and \eqref{eulerian:andre}.
\begin{cor} We have 
\begin{align}
\sum_{\sigma\in \S_n} q^{(\inv -\exc)\sigma}t^{\exc\sigma}
&=\sum_{k=0}^{\lfloor{(n-1)/2\rfloor}}\gamma_{n,k}(q^2,q)
t^{k}(1+t)^{n-1-2k}\label{gamma-q}\\
&=\sum_{k=0}^{\lfloor{(n-1)/2\rfloor}}(1+q)^k\and_{n, k}(q)
t^{k}(1+t)^{n-1-2k},\label{gamma-qd}
\end{align}
where 
$$
\and_{n, k}(q)=\sum_{\sigma\in \And_{n,k}}q^{2\res\sigma+\312\sigma}.
$$
\end{cor}

\medskip

By \eqref{q-ES-D} and Theorem~\ref{thm:main2} we derive  two interpretations for Han's $q$-Euler numbers\cite{Han19}.
\begin{cor} We have 
\begin{align}
E_{n}(q)&=\sum_{\sigma\in \And_{n}}q^{\res\sigma}\\
&=\sum_{\sigma\in \And_{n}}q^{\312\sigma-\des\sigma}.
\end{align}
\end{cor}
In Section~3 we shall give a triple sum formula for $D_n(1,q,t)$
 (cf. Theorem~\ref{formula(p=1)}) and
also a simple sum  formula for $D_n(1,-1,t)$ (cf. Theorem~\ref{formula(-1)}).
\section{Proof of main Theorems}
\subsection{x-factorization and MFS-action}
We recall  some definitions from \cite{FS73, FS74}.
A permutation  $w$ of a finite subset $\{a_1<a_2<\cdots <a_n\}$ of $\N$
 is a word $w=x_1\ldots x_n$.
The word $u$ obtained by juxtaposing two words $v$ and $w$ in
this order is written $u = v w$. The word $v$ (resp. $w$) is the left (resp. right)
factor of $u$. More generally, a factorization of length $q$ ($q>0$) of a word $w$
is any sequence $(w_1, w_2, ..., w_q)$ of words (some of them possibly empty)
such that the juxtaposition product $w_1 w_2 ... w_q$ is equal to $w$.
The following definition was given as a lemma in~\cite[Lemma~1]{FS74}.
\begin{defn}
Let $w = x_1x_2 ... x_n$ ($n > 0$) be a permutation and $x$ be one of
the letters $x_i$ ($1 < i< n$). Then $w$ has a unique factorization $(w_1 , w_2 , x, w_4, w_5)$ of length $5$, called its $x$-factorization, which is characterized by the three
properties
\begin{itemize}
\item[(i)] $w_1$ is empty or its last letter is less than $x$;
\item[(ii)] $w_2$ (resp. $w_4$) is empty or all its letters are greater than $x$;
\item[(iii)] $w_5$ is empty or its first letter is less than $x$.
\end{itemize}
\end{defn}
For instance, for $x=4$ the $x$-factorization of $w=76314582$ is
given by $(7631, \;\epsilon, \;4, \;58, \;2)$, where $\epsilon$ denotes the empty word.
It is not difficult to check the following facts about $x$-factorization (see \cite{FS76}).
\begin{prop}
Let $(w_1 , w_2 , x, w_4, w_5)$ be the $x$-factorization of a permutation. Then
\begin{itemize}
\item
$x$ is a peak of $\sigma$ iff $w_2=\epsilon$ and $w_4=\epsilon$,
\item
$x$ is a valley of $\sigma$ iff $w_2\neq\epsilon$ and $w_4\neq\epsilon$,
\item
$x$ is a double ascent of $\sigma$ iff $w_2=\epsilon$ and $w_4\neq\epsilon$,
\item
$x$ is a double descent of $\sigma$ iff $w_2\neq\epsilon$ and $w_4=\epsilon$.
\end{itemize}
\end{prop}
We can  charactrize Andr\'e permutations in terms of $x$-factorization~\cite{FS73}.
\begin{prop}\label{x-fact-andre}
A permutation $\sigma\in\S_n$ is an Andr\'e permutation  
if it is empty or satisfies the following:
\begin{itemize}
\item[(i)]
$\sigma$ has no double descents,
\item[(ii)]
$n-1$ is not a descent position, i.e. $\sigma_{n-1}<\sigma_n$,
\item[(iii)]
If $\sigma_i$ is a valley of $\sigma$ with  $\sigma_i$-factorization 
$(w_1, w_2, \sigma_i, w_4, w_5)$, then  $\min(w_2)>\min(w_4)$, i.e.,
the minimum letter of $w_2$ is larger than the minimum letter of $w_4$.
\end{itemize}
\end{prop}
\begin{proof}
For $\sigma\in\S_n$ satisfying  (i)--(iii),
let $\tau:=\sigma_{[k]}=\tau_1\ldots \tau_k$ for $1\leq k< n$.
\begin{itemize}
\item If $x:=\tau_i$ is  a double descent of $\tau$, then $x$ must be a valley in $\sigma$ by (i);
as all letters (if any)
 between $\tau_{i-1}$ and $x$ in $\sigma$  are larger than $x$, and all the letters between $x$ and $\tau_{i+1}$ are larger than $k$, the $x$-factorization of $\sigma$  is $(w_1,w_2,x,w_4,w_5)$ with $\min(w_4)>k$ and $\min(w_2)\leq\tau_{i-1}\leq k$, so $\min(w_4)>\min(w_2)$ which is impossible  by (iii).
\item  If $\tau_{k-1}$ is a descent in $\sigma_{[k]}$, then $\tau_{k}$ must be a valley in $\sigma$ because all the letters after $x$ in $\sigma$ are larger than $k$.
Let $x=\tau_k$, then the same argument as above leads to a  contradiction. 
\end{itemize}
Thus, 
 for all $1\leq k\leq n$, the restriction $\sigma_{[k]}$ is an Andr\'e permutation.
\end{proof}

We recall Br\"and\'en's modified Foata-Strehl action or MFS-action~\cite{Bra08}.
For   $x\in[n]$,  the MFS-action $\tilde{\varphi}_x$ on $\sigma\in \S_n$ with  
 $x$-factorization $(w_1,w_2,x,w_4,w_5)$ is defined as follows:
\[\tilde{\varphi}_x(\sigma)=
\left\{
\begin{array}{cl}
w_1w_4xw_2w_5 & \mbox{if $x$ is a double ascent or double descent,}\\
\sigma & \mbox{if $x$ is a valley or  peak.}
\end{array}\right.
\]
Thus $\tilde{\varphi}_x$ is an involution acting on $\S_n$ and it is not hard to see that $\tilde{\varphi}_x$ and $\tilde{\varphi}_y$ commute for all $x,y\in[n]$. Hence for any subset $X\subseteq [n]$ we may define the action 
$\tilde{\varphi}_X$ on $\sigma\in \S_n$ by
$$
\tilde{\varphi}_X(\sigma)=\prod_{x\in X}\tilde{\varphi}_x(\sigma).
$$
See Figure~\ref{valhop} for illustration, where  exchanging
$w_2$ and  $w_4$ in the $x$-factorisation is equivalent to
moving $x$  from a double ascent to  a double descent or vice versa. 

For $\sigma\in\S_n$ let $\mathsf{Orb}(\sigma)=\{g(\sigma):g\in\Z_2^n\}$ be the orbit of $\sigma$ under the action MFS-action. Let $\tilde{\sigma}$ be 
the  unique element in $\mathsf{Orb}(\sigma)$ without  double descents. The next theorem follows from the work of \cite{Bra08, FS76}.

\begin{figure}[t]
\setlength {\unitlength} {0.8mm}
\begin{picture} (90,40) \setlength {\unitlength} {1mm}
\thinlines
\put(24,8){\dashline{1}(1,0)(40,0)}
\put(64,8){\vector(1,0){0.1}}
\put(-22,14){\dashline{1}(-1,0)(32,0)}
\put(10,14){\vector(1,0){0.1}}
\put(6,18){\dashline{1}(-1,0)(-25,0)}
\put(-20,18){\vector(-1,0){0.1}}
\put(60,12){\dashline{1}(-1,0)(-30,0)}
\put(28,12){\vector(-1,0){0.1}}
\put(-35.5, 2.5){\circle*{1.3}}\put(-36.5, -1.5){$0$}
\put(-7,31){\line(-1,-1){28}}
\put(-7,31){\circle*{1.3}}
\put(20,4){\line(-1,1){27}}
\put(-7,32){$9$}
\put(-24,14){\circle*{1.3}}\put(-23,10){$5$}
\put(6,18){\circle*{1.3}}\put(9,16){$6$}
\put(20,4){\circle*{1.3}}
\put(20,4){\circle*{1.3}}\put(19.1,0){$1$}
\put(20,4){\line(1,1){24}}\put(44,28.5){\circle*{1.3}}
\put(24,8){\circle*{1.3}}\put(24.5,4.5){$3$}
\put(44,28.5){\line(1,-1){24}}\put(43,30){$8$}
\put(60.5,12){\circle*{1.3}}\put(62.5,11){$4$}
\put(68,4.5){\circle*{1.3}}\put(67.5,0){$2$}
\put(67.5,4.5){\line(1,1){16.5}}
\put(84,20.5){\circle*{1.3}}\put(83,22){$7$}
\put(84,21){\line(1,-1){17.5}}
\put(102, 2.5){\circle*{1.3}}\put(101,-1.5){$0$}
\end{picture}
\caption{MFS-actions on $\sigma=5\,9\,6\,1\,3\,8\,4\,2\,7$ with $\sigma(0)=\sigma(10)=0$ \label{valhop}}
\end {figure}
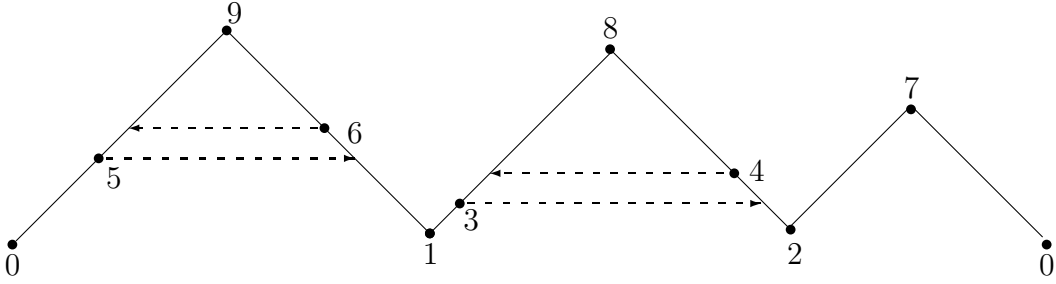

\begin{thm}\label{c1}
For any $\tilde{\sigma}\in\S_n$ without double decent, we have
$$
\sum_{\sigma\in \mathsf{Orb}(\tilde{\sigma})}p^{\res\sigma}q^{\312\sigma}
t^{\des\sigma}=p^{\res\tilde{\sigma}}
q^{\312\tilde{\sigma}}t^{\des\tilde{\sigma}}(1+t)^{n-1-2\des\tilde{\sigma}}.
$$
\end{thm}
\begin{proof}
For $\sigma\in\S_n$, consider a $\thot$ triple $(\sigma_j,\sigma_i,\sigma_{i+1})$, where $1\leq j<i\leq n$ and $\sigma_i<\sigma_j<\sigma_{i+1}$, where $(\sigma_i,\sigma_{i+1})$ is a pair of consecutive valley and peak, which means that there are no other peaks and valleys in between $\sigma_i$ and $\sigma_{i+1}$. The number of such triples is invariant under the action since $\sigma_i$ and
$\sigma_{i+1}$ cannot move and $a_j$ can not move over the peak $\sigma_{i+1}$. A similar reasoning applies to $\toht$.
\end{proof}
Let   $A_n(p,q,t,u,v,w)$  be   the generalized   Eulerian polynomials defined by
\begin{align}\label{eq:dfA}
A_n(p,q,t,u,v,w) := \sum_{\sigma\in \S_n} p^{\res \sigma} q^{\les \sigma}  t^{\des\sigma}u^{\da\sigma} v^{\dd\sigma}w^{\valley\sigma}.
\end{align}
As $\des=\valley+\dd$ we derive the 
following generalization of  \eqref{gamma-pq} from Theorem~\ref{c1}.
This was first proved in \cite{SZ12} by using combinatorial theory  of 
continued fractions.  
\begin{cor}\label{coro: gamma} For the $\gamma$-coefficients $\gamma_{n,k}(p,q)$ in 
\eqref{eq:combin} we have 
\begin{equation}\label{eq:gamma}
A_n(p,q,t,u,v,w)= \sum_{k=0}^{\floor{(n-1)/2}} \gamma_{n,k}(p,q) (tw)^k (u+vt)^{n-1-2k}.
\end{equation}
\end{cor}
\subsection{New action on permutations without double descent}
Let $\DD_n$ be the set of permutations of $[n]$ without double descent.
For any permutation $\sigma\in \DD_n$ and $x\in [n]$ we shall  identify 
 $\sigma$ with its $x$-factorization, i.e., $\sigma=(w_1, w_2, x, w_4, w_5)=w_1\, w_2\, x\, w_4\, w_5$, and 
 let  $y_1:=\min(w_2),\, y_2:=\min(w_4)$.
 A valley  $x$ of $\sigma$ is said to be
\begin{itemize}
\item \emph{good} (resp. \emph{bad})  if $y_1>y_2$ (resp. $y_1<y_2$); 
\item of   \emph{type I} 
if $\min(y_1,y_2)$ is a peak or double ascent,
\item of   \emph{type II} if $\min(y_1,y_2)$ is a valley.
\end{itemize}
We  denote by $\mathsf{Val}\,\sigma$ the set of valleys of 
$\sigma$.
\begin{prop}
Let  $\sigma\in \DD_n$ and $x\in \mathsf{Val}\,\sigma$ with $y=\min(y_1,y_2)$. 
\begin{itemize}
\item[(i)]
If  $y$ is a peak, then $w_4=y$ (resp. $w_2=y$) if  $y_1>y_2$ (resp. $y_1<y_2$).
\item[(ii)]
If  $y$ is a double ascent,  then $w_4=yw_4''$  (resp. $w_2=yw_2''$) 
with  $w_2'', w_4''\neq \epsilon$ if  $y_1>y_2$ (resp. $y_1<y_2$).
\item[(iii)]
If $y$ is a valley, then  $w_4=w_4'yw_4''$ (resp. $w_2=w_2'yw_2''$) with $w_2', w_2'', w_4', w_4''\neq \epsilon$ if  $y_1>y_2$ (resp. $y_1<y_2$).
\end{itemize}
\end{prop}
\begin{proof}  We assume that $y=y_2$.
\begin{itemize}
\item[(i)] If $y_2$ is a peak, then $w_4=y_2$ for, otherwise,  the word $w_4$ contains a letter next to $y_2$   and smaller than $y_2$.
\item[(ii)] If $y_2$ is a double ascent,  then  $w_4=yw_4'$  with  $w_4'\neq \epsilon$.
\item[(iii)] If $y_2$ is a valley, then  $w_4=w_4'yw_4''$ with $w_4', w_4''\neq \epsilon$.
\end{itemize}
The case $y=y_1$ can be proved similarly.
\end{proof}

\begin{defn}\label{def:action}
For $\sigma\in \DD_n$ and each $x\in \mathsf{Val}\,\sigma$ with $y=\min(y_1,y_2)$, we define its
 transform $\varphi(\sigma, x)$ as follows:
\begin{itemize}
\item[(i)]
If $y$ is a peak, then 
$$
\varphi(\sigma, x)=\begin{cases}
(w_1,y,x,w_2,w_5)& \textrm{if}\quad y=y_2,\\
(w_1,w_4,x, y,w_5)&\textrm{if}\quad y=y_1.
\end{cases}
$$
\item[(ii)]
If $y$ is a double ascent,  then  
$$
\varphi(\sigma, x)=\begin{cases}
(w_1,yw_2,x,w'',w_5)& \textrm{if}\quad y=y_2\; \text{and} \; w_4=yw'',\\
(w_1,w'',x, yw_4,w_5)&\textrm{if}\quad y=y_1\; \text{and} \; w_2=yw''.
\end{cases}
$$
%
\item[(iii)]
If $y$ is a valley, then   
$$
\varphi(\sigma, x)=\begin{cases}
(w_1,w_2yw',x,w'', w_5)& \textrm{if}\quad y=y_2\; \text{and} \; w_4=w'yw'',\\
(w_1,w',x, w''yw_4,w_5)&\textrm{if}\quad y=y_1\; \text{and} \; w_2=w'yw''
\end{cases}
$$
with $w', w''\neq \epsilon$.
\end{itemize}
\end{defn}
Obviously this transformation switches $y$ from left  to right or  right to left of $x$
and  $\varphi(\varphi(\sigma, x),x)=\sigma$.
The three cases (i),(ii),(iii) of  the transformation with $y=y_2$ are  depicted  in  Figure~\ref{V}.

We record the basic properties of this transformation in the following proposition.

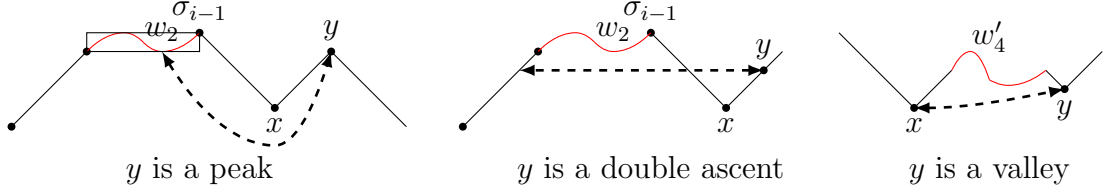
\begin{figure}\label{V}
\begin{tikzpicture}[scale=0.5]
\draw (-17,2)--(-15,0)--(-13.5,1.5);
\node[below] at (-15,0) {$x$};
\fill (-15,0) circle (3pt);
\fill (-13.5,1.5) circle (3pt);
\draw (-13.5,1.5)--(-11.5,-0.5);
\fill (-17,2) circle (3pt);
\node[above] at (-13.5,1.5) {$y$};
\draw (-22,-0.5)--(-20,1.5);
\fill (-22,-0.5) circle (3pt);
\fill (-20,1.5) circle (3pt);
\draw[color=red] (-20,1.5) parabola bend (-19,2) (-18.5,1.75);
\draw[color=red] (-18.5,1.75) parabola bend (-18,1.5) (-17,2);
\node[above] at (-17,2) {$\tiny{\sigma_{i-1}}$};
\node[above] at (-18,1.5) {$w_2$};
\node[below] at (-17,-1) {$y$ is a peak};
\draw[>=latex,<->,dashed,line width=1pt] (-18,1.5) parabola bend (-15,-1) (-13.5,1.5);
\draw (-20,1.5) rectangle (-17,2);
\draw (-5,2)--(-3,0)--(-1.5,1.5);
\node[below] at (-3,0) {$x$};
\fill (-3,0) circle (3pt);
\fill (-2,1) circle (3pt);
\fill (-5,2) circle (3pt);
\node[above] at (-2,1) {$y$};
\draw (-10,-0.5)--(-8,1.5);
\fill (-10,-0.5) circle (3pt);
\fill (-8,1.5) circle (3pt);
\draw[color=red] (-8,1.5) parabola bend (-7,2) (-6.5,1.75);
\draw[color=red] (-6.5,1.75) parabola bend (-6,1.5) (-5,2);
\node[above] at (-5,2) {$\sigma_{i-1}$};
\draw[>=latex,<->,dashed,line width=1pt] (-2,1)--(-8.5,1);
\node[above] at (-6,1.5) {$w_2$};
\node[below] at (-5,-1) {$y$ is a double ascent};
\draw (0,2)--(2,0)--(3,1);
\node[below] at (2,0) {$x$};
\fill (2,0) circle (3pt);
\fill (6,0.5) circle (3pt);
\draw (5.5,1)--(6,0.5)--(7,1.5);
\node[below] at (6,0.5) {$y$};
\draw[color=red] (3,1) parabola bend (3.5,1.5) (4,0.75);
\draw[color=red] (4,0.75) parabola bend (4.5,0.6) (5.5,1);
\node[above] at (4,1.2) {$w_4'$};
\draw[>=latex,<->,dashed,line width=1pt] (2,0) cos (6,0.5);
\node[below] at (4,-1) {$y$ is a valley};
\end{tikzpicture}
\caption{The transform $\varphi(x,\sigma)$ according to the type of $y$}
\end{figure}

\begin{prop}\label{Q2}
If  $\sigma\in\DD_{n,k}$ and $x\in \mathsf{Val}\,\sigma$,  then $\varphi(\sigma, x)\in \DD_{n,k}$ 
and 
\begin{align}\label{eq:123x}
\begin{split}
\res\varphi(\sigma, x)&=\begin{cases}
\res\sigma+1& \text{if $x$ is good} \\
\res\sigma-1& \text{if $x$ is bad};
\end{cases}\\
\les\varphi(\sigma, x)&=\begin{cases}
\les\sigma-1& \text{if $x$ is good}\\
\les\sigma+1& \text{if $x$ is bad}.
\end{cases}
\end{split}
\end{align}
\end{prop}
\begin{proof} We assume that $y:=\min(y_1,y_2)=y_2$. By definition~\ref{def:action},  
the descent number is unchanged after the transform, i.e., 
$\des\sigma=\des\varphi(\sigma, x)$.  Next  we  verify that 
$\varphi(\sigma, x)$ has no double descent for the above three cases.
For (i), 
the only possible double descent in  $\varphi(\sigma, x)=(w_1,y,x,w_2,w_5)$ is  $\last(w_2)$, but this is impossible for, otherwise, the same letter $\last(w_2)$ would be a double descent in $\sigma$. The case (ii) is clear. For case (iii),  the only possible double descent in  $\varphi(\sigma, x)=(w_1,w_2yw_4',x,w_4'',w_5)$ is  $\last(w_4')$, but this is also impossible for, otherwise, the same letter $\last(w_4')$ would be a double descent in $\sigma$. Thus
$\varphi(\sigma, x)\in \DD_{n,k}$.

It remains to prove \eqref{eq:123x}. 
For a permutation $\sigma\in \S_n$,
it is convenient to say that  a $\312$ pattern $(\sigma_j,\sigma_{j+1},\sigma_i)$ is  created by $\sigma_i$, and  a $\res$ pattern $(\sigma_i,\sigma_j,\sigma_{j+1})$ is created by $\sigma_i$. Let $x\in \mathsf{Val}\sigma$
and $(w_1,w_2,x,w_4,w_5)$ be the $x$-factorization of $\sigma$
and $y:=\min(w_4)$. We examine three cases corresponding to the above (i)-(iii), respectively.
\begin{itemize}
\item[(i)]
If $y$ is a peak, then  $\tau=(w_1,y,x,w_2,w_5)$.  Let  $a\in [n]$.
  \begin{itemize}
 \item  If $a=y$ and creats  the $\312$-pattern $(\last(w_2),x,y)$ in $\sigma$, then $a$  creats 
 the  $\res$-pattern  $(y,x, \first(w_2))$ in $\tau$.
 \item If $a=x$ and creats the $\res$-pattern $(x, y, \first(w_5))$ in $\sigma$, then
 $a$ creats the same pattern $(x, \last(w_2), \first(w_5))$ in $\tau$.
 \item If $a\neq x,y$ and creats the $\res$-pattern $(a, \last(w_1), \first(w_2))$ in $\sigma$, then 
 $a$ creats the same pattern $(a, \last(w_1), y)$ in $\tau$.
\item If $a\neq x,y$ and creats the $\312$-pattern $(y,\first(w_5),a)$ in $\sigma$, then $a$ creats the same pattern
$(\last(w_2),\first(w_5),a)$ in $\tau$.
\end{itemize}
Hence,  if $a\in[n]\setminus\{y\}$,  the number of $\les$-patterns (resp. $\res$-patterns)
  created by $a$ is  unchanged under the action.
\item[(ii)]
If $y$ is a double ascent, then 
$\tau=(w_1,yw_2,x,w_4',w_5)$ with  $w_4=yw_4'$ and $w_4' \neq \epsilon$.
If $y$ creats  the $\les$-pattern $(\last(w_2),x,y)$ in $\sigma$, then
$y$ creats the $\res$-pattern $(y,x,\first(w_4))$ in $\tau$. 
As in (i) we can verify that the number of patterns created by other entry $a\in [n]$  will  not change.
\item[(iii)]
If $y$ is a valley, then  $\tau=(w_1,w_2yw_4',x,w_4'',w_5)$ with $w_4=w'yw_4'$ and 
$w_4', w_4''\neq \epsilon$.
If $y$ creats the $\les$-pattern $(\last(w_2),x,y)$ in $\sigma$, then $y$ creats the $\res$-pattern $(y,x,\first(w_4^{''}))$ in $\tau$.
 Again as in case (i) for other entry $a\in [n]$, the number of $\les$ and $\res$ created by $a$ will not  change.
\end{itemize}
The proof is thus completed.
\end{proof}

Next we define the transform $\varphi(\sigma, S)$ for  any subset $S$ of $\mathsf{Val}(\sigma)$ with
 $\sigma\in \DD_n$.
\begin{defn} Let $\sigma\in \DD_n$.
For  any $S\subseteq \mathsf{Val}\,\sigma$, 
let  $\{S_1,S_2\}$ be the partition of $S$ such that  
\begin{itemize}
\item[(1)]
$S_1$ is the subset of $S$  consisting of valleys of type I, say $i_1, \ldots, i_l$;
\item[(2)] 
$S_2$  is  the subset of $S$ consisting of valleys of type II, say $j_k<\cdots< j_2<j_1$.
\end{itemize}
Define the transforms 
\begin{align*}
\varphi(\sigma,S_1)&=\varphi(i_l, \ldots, \varphi(i_2, \varphi(i_1, \sigma))),\\
\varphi(\sigma,S_2)&=\varphi(j_k, \ldots, \varphi(j_2, \varphi(j_1, \sigma))),\\
\varphi( \sigma,S)&=\varphi(\varphi(\sigma, S_1), S_2).
\end{align*}
\end{defn}
\begin{rem}
 The  image $\varphi(\sigma, S_1)$ is independent of the order  of $i_1, \ldots, i_l$ while 
 $\varphi(\sigma, S_2)$ is defined for the elements of $S_2$ in the decreasing order  $j_1>j_2> \ldots >j_1$.
\end{rem}
In what follows, if $w=w_1\ldots w_n$ is a permutation of $[n]$ and  $u$ a subword of $w$, i.e., $u=w_iw_{i+1}\ldots w_j$ ($1\leq i\leq j\leq n$),   we write $u\subseteq w$ and $u\in w$ if $u$ is a single letter $w_i$.
\begin{prop}\label{Q3}
If  $\sigma\in\And_{n,k}$ and $S\subseteq \mathsf{Val}(\sigma)$, 
then $\tau:=\varphi(\sigma,S)\in \DD_{n,k}$ is well defined and 
\begin{align}\label{eq:s-tau}
S=\{x\in \mathsf{Val}(\tau)\,|\, \textrm{$x$ is a bad guy}\}.
\end{align}
\end{prop}
\begin{proof}
For $\sigma\in\DD_{n,k}$ and $x_1, x_2\in S$, 
let $(w_1,w_2,x_1,w_4,w_5)$ and  $(w_1',w_2',x_2,w_4',w_5')$ be the
$x_1$-factorization and $x_2$-factorization of $\sigma$, respectively.
Assume that  $x_1<x_2$, hence 
 $x_1\in w_1'$ or $x_1\in w_5'$.  Thus, if $w_2x_1 w_4\cap w_2'x_2w_4'\neq \emptyset$, then 
$w_2'x_2w_4'\subseteq w_4$ or $w_2'x_2w_4'\subseteq w_2$, namely,  one of the following 
holds:
$$
w_2'x_2w_4'\subseteq w_1, \quad w_2'x_2w_4'\subseteq w_2, \quad 
w_2'x_2w_4'\subseteq w_4, \quad w_2'x_2w_4'\subseteq w_5.
$$ 
Thus when we apply $\varphi$ to the elements of $S_1$ the image $\varphi(\sigma, S_1)$ is independ from the 
order of elements, and $\varphi(\sigma, S_2)$ is well defined if we apply $\varphi$ to the elements of $S_2$ 
in decreasing order.  As for each $\sigma\in \And_{n,k}$ and any $x\in \mathsf{Val}(\sigma)$ there holds
$y_1>y_2$,  and for each element of $S$ the application of $\varphi$ switches $y_2$ from  right to left of $x$, we have \eqref{eq:s-tau}.
\end{proof}

For any set $S$ we denote by $ 2^{S} $  the set of all subsets of $S$. 
In what follows,
for $\sigma\in  \DD_{n,k}$ we will identify  $\mathsf{Val}(\sigma)$ with  $[k]$ under the map
$a_i\mapsto i$ for $i\in [k]$ if $\mathsf{Val}(\sigma)$ consists of $a_1<a_2<\ldots<a_k$,  and identify 
any subset $S\in  \mathsf{Val}(\sigma)$ with its image $S'\in 2^{[k]}$. Thus we will use $2^{[k]}$ instead of 
$2^{ \mathsf{Val}(\sigma)}$.
\begin{prop}\label{key1}
The map 
$\varphi:   \And_{n,k}\times 2^{[k]}\to \mathcal{G}_{n,k}$ is  a bijection 
 such that   for $(\sigma,S)\in   \And_{n,k}\times 2^{[k]}$ we have 
\begin{equation} \label{prop}
\begin{split}
\res\sigma +|S|& = \res\varphi(\sigma, S)\\
\312\sigma -|S|& = \312\varphi(\sigma, S).
\end{split}
\end{equation}
\end{prop}
\begin{proof}
By Propositions~\ref{Q2} and \ref{Q3}, for $(\sigma, S)\in \And_{n,k}\times  2^{[k]}$ the image
$\tau:=\varphi(\sigma, S)$ is an element in  $\mathcal{G}_{n,k}$ and satisfies \eqref{eq:s-tau} and \eqref{prop}.
To show that $\varphi$ is bijective,
we construct   the reverse of $\varphi$. 
For any $\tau\in\DD_{n,k}$,  let $\mathsf{Val}^{(1)}(\tau)$ and $\mathsf{Val}^{(2)}(\tau)$ be the sets of valleys of $\tau$ of types	 I and II, respectively, and  define
\begin{align*}
S_2(\tau):&=\{x\in \mathsf{Val}^{(2)}(\tau)\,|\,  \textrm{$x$ is a bad guy}\};\\
S_{1}(\tau):&=\{x\in \mathsf{Val}^{(1)}(\tau)\,|\,  \textrm{$x$ is a bad guy}\}.
\end{align*}
Note that $\varphi(\sigma, \emptyset)=\sigma$ and $\sigma\in \And_{n,k}$ if and only if $S_{1}(\sigma)\cup S_2(\sigma)=\emptyset$.
 We recover the pair 
$(\sigma, S)\in \And_{n,k}\times  2^{[k]}$ such that $\varphi(\sigma, S)=\tau$
by applying the 
following algorithm:
\begin{itemize}
\item[(i)]  Input $(\sigma, S):=(\tau, \emptyset)$.
\item[(ii)] While  $S_2(\sigma):=\{x\in \mathsf{Val}^{(2)}\sigma\,|\,  y_1<y_2\}\neq \emptyset$ do 
$(\sigma, S):=(\varphi(\sigma,z), S\cup\{z\})$ with $z:=\min{S_2(\sigma)}$.
\item[(iii)] Let $(\sigma, S)$ be the output of (ii) with $S_{1}(\sigma):=\{x\in \mathsf{Val}^{(1)}\sigma\,|\,  y_1<y_2\}$, and 
$$(\sigma, S):=(\varphi(\sigma,S_1(\sigma)), S\cup S_{1}(\sigma)).$$
\end{itemize}
To see  that  the loop (ii) is finite we just need to verify (easy!) that  $z$ is a good guy in
$\varphi(\sigma, z)$, which implies that 
if $z'=\min S_2(\varphi(\sigma, z))$ then
$z'>z$. 
It is clear that $S=S_1(\tau)\cup S_2(\tau)$ and $\varphi(\sigma, S)=\tau$. 
\end{proof}


For the reader's  convenience, we  run the map $\varphi: \And_{5,2} \times 2^{[2]}\to \DD_{5,2}$ in Figure~\ref{fig:52}, and give one  example 
for  $\varphi: \And_{5,2} \times 2^{[2]}\to \DD_{5,2}$  and   $\varphi^{-1}: \DD_{13,5}\to \And_{13,5}\times 2^{[5]}$, respectively.
\begin{itemize}
\item[(A)] If $\sigma=31524\in\And_{5,2}$ and $S=\{1,2\}$, then 
\begin{itemize}
\item for the  valley $1$, $\min(y_1,y_2)=2$ is a valley, so $1$ is a type II valley.
\item for the valley $2$,  $\min(y_1,y_2)=4$ is a peak, so $2$ is a type I valley.
\end{itemize}
So, we should first deal with the valley $2$, the $2$-factorization is 
$(w_1,\,w_2,\,x,\,w_4,\,w_5)=(31,5,\,2,\,4,\,\emptyset)$ 
according to case $(i)$ of $\varphi$, we just  exchange $4$ and $5$, and 
get $31425$; then we apply $\varphi$ to   the valley $1$ in $31425$, 
the $1$-factorization is $(w_1,\,w_2,\,x,\,w_4,\,w_5)=(\emptyset,3,\,1,\,425,\,\emptyset)$,  
which is case $(iii)$ of $\varphi$,  we just exchange $1$ and $2$, and  get $\varphi(\sigma,S)=32415$.

\item[(B)]
For
$\tau=11\,\mathbf{2}\,12\,13\,\mathbf{1}\,6\,\mathbf{4}\,5\,\mathbf{3}\,8\,9\,\mathbf{7}\,10\in\DD_{13,5}$.
\begin{itemize}
\item[(i)] First let $(\sigma, S):=(\tau, \emptyset)$.  We have  $S_{2}=\{1,3\}$ 
and $\min(S_{2})=1$, so $S=\{1\}$ and 
$$
\sigma:=\varphi(\sigma, 1)=11\,\mathbf{1}\,12\,13\,\mathbf{2}\,6\,\mathbf{4}\,5\, \mathbf{3}\,8\,9\,\mathbf{7}\,10.
$$
As  $S_{2}=\{3\}$, we have $\min S_{2}=3$ and 
$S:=\{1,3\}$, hence
$$
\sigma:=\varphi(\sigma,3)=11\,\mathbf{1}\,12\,13\,\mathbf{2}\,6\,\mathbf{3}\,5\, \mathbf{4}\,8\,9\,\mathbf{7}\,10.
$$
\item[(ii)]
Now for $\sigma$, $S_{2}=\emptyset$ and $S_1=\{4,7\}$, then $S:=\{1,3,4,7\}$ and 
$$
\sigma:=\varphi(\sigma, S_1)=11\,\mathbf{1}\,12\,13\,\mathbf{2}\,6\,\mathbf{3}\,10\,\mathbf{7}\,8\, 9\,\mathbf{4}\,5\in\And_{13,5}.
$$ 
We can check  that $\varphi(\sigma, S)=\tau$.
\end{itemize}
\end{itemize}

\begin{figure}[t]
\begin{tabular}{|c|c|c|c|}
\hline
$\sigma\in\And_{5,2}$ & $(2-13)$ & $(31-2)$&$S\in 2^{\mathsf{Val}\sigma}$\\\hline
$31524$ & $2$ & $2$&$\emptyset$\\\hline
$3\mathbf{1}524$ & $2$ & $2$&$\{1\}$\\\hline
$315\mathbf{2}4$ & $2$ & $2$&$\{2\}$\\\hline
$3\mathbf{1}5\mathbf{2}4$ & $2$ & $2$&$\{1,2\}$\\\hline\hline
$41523$ & $1$ & $3$&$\emptyset$\\\hline
$4\mathbf{1}523$ & $1$ & $3$&$\{1\}$\\\hline
$415\mathbf{2}3$ & $1$ & $3$&$\{2\}$\\\hline
$4\mathbf{1}5\mathbf{2}3$ & $1$ & $3$&$\{1,2\}$\\\hline\hline
$51423$ & $0$ & $4$&$\emptyset$\\\hline
$5\mathbf{1}423$ &$0$ & $4$&$\{1\}$\\\hline  
$514\mathbf{2}3$ & $0$ & $4$&$\{2\}$\\\hline
$5\mathbf{1}4\mathbf{2}3$ & $0$ & $4$&$\{1,2\}$\\\hline\hline
$53412$ & $0$ & $2$&$\emptyset$\\\hline
$534\mathbf{1}2$ & $0$ & $2$&$\{1\}$\\\hline
$5\mathbf{3}412$ & $0$ & $2$&$\{3\}$\\\hline
$5\mathbf{3}4\mathbf{1}2$ &  $0$ & $2$&$\{1,3\}$\\\hline
\end{tabular}
\hspace{-0.3cm}
\begin{tabular}{|c|c|c|}
\hline
$\tau\in\DD_{5,2}$ & $(2-13)$ & $(31-2)$\\\hline
$3\mathbf{1}5\mathbf{2}4$ & $2$ & $2$\\\hline
$3\mathbf{2}5\mathbf{1}4$ & $3$ & $1$\\\hline
$3\mathbf{1}4\mathbf{2}5$ & $3$ & $1$\\\hline
$3\mathbf{2}4\mathbf{1}5$ & $4$ & $0$\\\hline\hline
$4\mathbf{1}5\mathbf{2}3$ & $1$ & $3$\\\hline
$4\mathbf{2}5\mathbf{1}3$ & $2$ & $2$\\\hline
$4\mathbf{1}3\mathbf{2}5$ & $2$ & $2$\\\hline
$4\mathbf{2}3\mathbf{1}5$ & $3$ & $1$\\\hline\hline
$5\mathbf{1}4\mathbf{2}3$ & $0$ & $4$\\\hline
$5\mathbf{2}4\mathbf{1}3$ & $1$ & $3$\\\hline  
$5\mathbf{1}3\mathbf{2}4$ & $1$ & $3$\\\hline
$5\mathbf{2}3\mathbf{1}4$ & $2$ & $2$\\\hline\hline
$5\mathbf{3}4\mathbf{1}2$ & $0$ & $2$\\\hline
$2\mathbf{1}5\mathbf{3}4$ & $1$ & $1$\\\hline
$4\mathbf{3}5\mathbf{1}2$ & $1$ & $1$\\\hline
$2\mathbf{1}4\mathbf{3}5$ & $2$ & $0$\\\hline
\end{tabular}
\caption{The bijection $\varphi: (\sigma, S)\mapsto \tau$ from $ \And_{5,2}\times 2^{[2]}$ to $\DD_{5,2}$}
\label{fig:52}
\end{figure}
\subsection{Proof of Theorem~\ref{thm:main1}}
Clearly \eqref{gamma-pq}  is a special case of Corollary~\ref{coro: gamma}, and 
\eqref{gamma-d} is equivalent to
\begin{align}\label{eq:equiv}
(p+q)^k \,\sum_{\sigma\in \And_{n,k}}p^{\res\sigma}q^{\les\sigma-k}=
\sum_{\sigma\in \DD_{n,k}}p^{\res\sigma}q^{\les\sigma}.
\end{align}
As $(p+q)^k=\sum_{S\in 2^{[k]}}p^{|S|}q^{k-|S|}$ we can rewrite the above identity as
$$
\sum_{(\sigma,S)\in  \And_{n,k}\times 2^{[k]} }p^{\res\sigma+|S|}q^{\les\sigma-|S|}=
\sum_{\sigma\in \DD_{n,k}}p^{\res\sigma}q^{\les\sigma}.
$$
The result  follows  from Proposition~\ref{key1}.
\qed

\subsection{Proof of Theorem~\ref{thm:main2}}
We shall use 
the J-type continued fraction  as a  formal power series defined by
 $$
\sum_{n=0}^\infty \mu_nt^n=\cfrac{1}{1-b_0 t-\cfrac{\lambda_1 t^2}
{1-b_1 t-\cfrac{\lambda_2t^2}{1-\cdots}}},
 $$ 
 where $(b_n)$ and $(\lambda_{n+1})$ ($n\geq 0$) are two sequences in a commutative ring.
When $b_n=0$ we obtain the S-type  continued fraction:
$$
\sum_{n=0}^\infty \mu_nt^n=\cfrac{1}{1-\cfrac{\lambda_1 t}{1-\cfrac{\lambda_2 t}{1-\cdots}}}.
 $$ 

Recall   the following continued fraction expansion formula from  \cite[(28)]{SZ12}:
\begin{equation} \label{eq:master}
\begin{split}
\sum_{n\ge1}& A_n(p,q,t,u,v,w)x^{n-1}=\\
&\dfrac{1}
{1-(u+tv)[1]_{p,q}x-\dfrac{[1]_{p,q}[2]_{p,q}twx^2}
{1-(u+tv)[2]_{p,q}x-\dfrac{[2]_{p,q}[3]_{p,q}twx^2}
{\cdots}}}
\end{split}
\end{equation}
with $b_n=(u+tv)[n+1]_{p,q}$ and $\la_n=[n]_{p,q}[n+1]_{p,q}tw$.

By Theorem~\ref{thm:main1} and  substituting  $(t,u,v, w)$ with $(p+q, 0, 1,t)$ in \eqref{eq:gamma}, 
we obtain
$$
A_n(p,q,p+q, 0, 1,t)=(p+q)^{n-1}\sum_{k=0}^{\lfloor (n-1)/2\rfloor}d_{n,k}(p, q)t^k.
$$
Thus, substituting  $(t,u,v, w)$ with $(p+q, 0, 1,t)$ in \eqref{eq:master} and replacing $x$ by $x/(p+q)$ we obtain
the same continued fraction in \eqref{pq-shift}.  This proves \eqref{eq:1}.
\qed
\section{Two explicit formulae}

First we derive a formula  for $D_n(1,q,t)$ from 
the work of Shin-Zeng~\cite{SZ12}, Han-Mao-Zeng~\cite{HMZ} and 
Josuat-Verg\`es~\cite{JV11}.
\begin{thm}\label{formula(p=1)} For $n\geq 1$ we have 
\begin{align*}
D_n(1,q,t)
&=\frac{1}{v(1-q)}\left(\frac{1+u}{(1+uv)(1-q^2)}\right)^{n-1}\\
&\times \sum_{k=0}^n(-1)^k\left(\sum_{j=0}^{n-k}v^j\binom{n}{j}\binom{n}{j+k}-\binom{n}{ j-1}
\binom{n}{j+k+1}\right)\cdot \left(\sum_{i=0}^kv^iq^{i(k+1-i)}\right),\nonumber
\end{align*}
where 
\begin{align}
u=&\frac{1+q^2-2(1+q)t-(1+q)\sqrt{(1+q)^2-4(1+q)t}}{2(q-t(1+q))}\label{u},\\
v=&\frac{(1+q)-2t-\sqrt{(1+q)^2-4t(1+q)}}{2t}\label{v}.
\end{align}
\end{thm}
\begin{proof} 
By \eqref{eq:gamma} we have 
\begin{align}
A_n(1,q,q, 1,1,  t(1+q^{-1}))=(1+q)^{n-1}D_n(1, q,t).
\end{align}
From Corollary~3.2 in~\cite{HMZ} we derive
$$
A_n(1,q,q,1,1, t(1+q^{-1}))=\left(\frac{1+u}{1+uv}\right)^{n-1}\sum_
{\sigma\in\S_n}v^{\des\sigma}q^{\les\sigma},
$$
where $u$ and $v$ are given by \eqref{u} and \eqref{v}.
By Theorem~6.3 in~\cite{JV11}, we have 
\begin{align*}
\sum_
{\sigma\in\S_n}y^{\des\sigma}q^{\les\sigma}&=
\sum_{\sigma\in\S_n}y^{\asc\sigma}q^{(13-2)\sigma}\\
&=\frac{1}{y(1-q)^n}\sum_{n=0}^n(-1)^k\left(\sum_{j=0}^{n-k}y^j\binom{n}{ j}\binom{n}{ j+k}-\binom{n}{j-1}\binom{n}{ j+k+1}\right)\\
&\hspace{3cm}\times \left(\sum_{i=0}^ky^iq^{i(k+1-i)}\right).
\end{align*}
Putting the above three formulae together  completes the proof.
\end{proof}
A \emph{Motzkin path}  of length  $n$ is a sequence of points 
$\eta:=(\eta_0, \ldots, \eta_n)$  in the integer plane 
$\mathbb{Z}\times\mathbb{Z}$  such that 
\begin{itemize}
\item $\eta_0=(0,0)$ and $\eta_n=(n,0)$,
\item $\eta_i-\eta_{i-1}\in \{(1,0), (1, 1), (1,-1)\}$,
\item $\eta_i:=(x_i, y_i)\in \N\times \N$ for $i=0,\ldots, n$.
\end{itemize}
In other words, a Motzkin path of length $n$ is a lattice path starting at $(0,0)$, ending at $(n,0)$, 
and never going below the $x$-axis, consisting of up-steps $\mathsf{U}=(1,1)$, level-steps $\mathsf{L}=(1,0)$, and down-steps
$\mathsf{D}=(1,-1)$. 
Let $\mathcal{MP}_n$ be the set of Motzkin paths of length $n$.
Clearly we can  identify 
 Motzkin paths of length $n$ with words $w$ on 
 $\{\mathsf{U, L, D}\}$ of length $n$ such that all prefixes of $w$ 
 contain no more $\mathsf{D}$'s than $\mathsf{U}$'s and the number of $\mathsf{D}$'s equals the number of $\mathsf{D}'s$.
The height of a step $(\eta_i, \eta_{i+1})$ is the coordinate of the starting point $\eta_i$.
Given a Motzkin path $p\in \mathcal{MP}_n$ and two sequences 
$(b_i)$ and $(\lambda_i)$  of a commutative ring $\mathrm{R}$, 
we weight   
each up-step by 1, and each level-step (resp. down-step)   at height $i$ by 
$b_i$ (resp. $\lambda_i$) and  define the weight $w(p)$ of $p$  by 
  the product of the weights of all its steps. 
  The following result of Flajolet~\cite{Fl80} is our starting point.

\begin{lem}[Flajolet]\label{flajolet} We have 
$$
\sum_{n=0}^{\infty}\left(\sum_{p\in \mathcal{MP}_n}w(p)\right)t^n
=\cfrac{1}{1-b_0t-\cfrac{\lambda_1t^2}{1-b_1t-\cfrac{\lambda_2t^2}{1-b_2t-\cdots}}}.
$$
\end{lem}

A Motzkin path without level-steps is called a \emph{Dyck path}, 
and a Motzkin path without level-steps at odd height is called
an \emph{Andr\'e path}.  
We denote by $\mathcal{AP}_{n, k}$ the set of Andr\'e paths of half-length $n$ with $k$ level-steps, 
and 
$\mathcal{DP}_n$ the set of Dyck paths of half length $n$.
\begin{lem}\label{Euler-Flajolet} Let $b_{i}=0$  ($i\geq 0$) 
and $\lambda_{i}=\lfloor \frac{i+1}{2}\rfloor$ ($i\geq 1$). Then 
 $$
 n!=\sum_{p\in \mathcal{MP}_n}w(p).
 $$
 In other words,  the factorial $n!$ is the generating polynomial of $\mathcal{DP}_n$.
\end{lem}
\begin{proof} Recall
the following formula of Euler:
\begin{align}\label{euler}
\sum_{n\geq 0}n!x^n&=\cfrac{1}{
1 - \cfrac{1 x}{
1 - \cfrac{1x}{
1-\cfrac{2x}{
1-\cfrac{2x}{
\cdots}}}}}
\end{align}
with $\lambda_{n}=\lfloor{(n+1)/2}\rfloor$.
The result then follows from Lemma~\ref{flajolet}.
\end{proof}
\begin{rem}
A bijective proof of Euler's formula \eqref{euler}  is known, see  \cite[(4.9)]{PZ19}.
\end{rem}
\begin{lem}\label{lem1} Let $b_{2i}=1$, $b_{2i+1}=0$  ($i\geq 0$) 
and $\lambda_{k}=\lfloor \frac{k+1}{2}\rfloor {t}$ ($i\geq 1$). Then 
 $$
 D_{n+1}(1,-1, t)=\sum_{p\in \mathcal{AP}_n}w(p).
 $$
 In other words,  the polynomial $D_{n+1}(1,-1, t)$ is the generating polynomial of 
 Andr\'e paths of length $n$.
\end{lem}
\begin{proof}
When $(p,q)=(1,-1)$ formula\eqref{pq-shift} reduces to
\begin{align}
\sum_{n=0}^\infty D_{n+1}(1,-1, t) x^n=
\cfrac{1}{1-x-\cfrac{tx^2}{1-\cfrac{tx^2}{1-x-\cfrac{2tx^2}{1-\cfrac{2tx^2}{1-x-\cdots}}}}}
\end{align}
with coefficients $b_{2i}=1$, $b_{2i+1}=0$ 
and $\lambda_{i}=\lfloor \frac{i+1}{2}\rfloor{t}$.  The result follows from Lemma~\ref{flajolet}.
\end{proof}

 Let 
$$
\mathcal{Y}_{n,k}:=\{(y_1, \ldots, y_{k+1})\in \N^{k+1}: y_1+\cdots +y_{k+1}=n-2k \}.
$$

\begin{ex}
An illustration   of the bijection $\psi$ is given in Figure~\ref{fig:motzkin}.
\begin{figure}[t]
\begin{tikzpicture}[scale=0.5]
\draw[line width=1pt] (-20,0)->(-19,0);
\fill (-20,0) circle (3pt);
\fill (-19,0) circle (3pt);
\draw[line width=1pt](-19,0)->(-16,3);
\fill (-18,1) circle (3pt);
\fill (-17,2) circle (3pt);
\fill (-16,3) circle (3pt);
\draw[line width=1pt] (-16,3)->(-15,2);
\fill (-15,2) circle (3pt);
\draw[line width=1pt] (-15,2)->(-13,2);
\fill (-14,2) circle (3pt);
\fill (-13,2) circle (3pt);
\draw[line width=1pt] (-13,2)->(-12,3);
\fill (-12,3) circle (3pt);
\draw[line width=1pt] (-12,3)->(-11,2);
\fill (-11,2) circle (3pt);
\draw[line width=1pt] (-11,2)->(-10,2);
\fill (-10,2) circle (3pt);
\draw[line width=1pt] (-10,2)->(-8,0);
\fill (-9,1) circle (3pt);
\fill (-8,0) circle (3pt);
\draw[line width=1pt] (-20,0)->(-8,0);
\node[scale=2,above] at (-6,1) {$\longrightarrow$};
\node[scale=1.5,above] at (-6,1.75) {$\psi$};
\node[scale=0.9,left] at (0.5,0) {$\bigl((1,0,2,1,0),$};
\fill (1,0) circle (3pt);
\draw[line width=1pt](1,0)->(4,3);
\fill (2,1) circle (3pt);
\fill (3,2) circle (3pt);
\fill (4,3) circle (3pt);
\draw[line width=1pt] (4,3)->(5,2);
\fill (5,2) circle (3pt);
\draw[line width=1pt] (5,2)->(6,3);
\fill (6,3) circle (3pt);
\draw[line width=1pt] (6,3)->(9,0);
\fill (7,2) circle (3pt);
\fill (8,1) circle (3pt);
\fill (9,0) circle (3pt);
\draw[line width=1pt] (1,0)->(9,0);
\node[scale=0.9,right] at (9,0) {$\bigr)$};
\end{tikzpicture}
\caption{The bijection
$\psi: \mathcal{AP}_{12, 4}\to \mathcal{Y}_{12,4}\times \mathcal{DP}_{4}$}
  \label{fig:motzkin}
\end{figure}
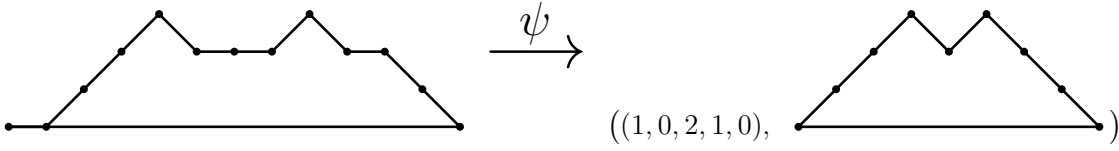
\end{ex}

\begin{lem}\label{lem2} For $0\leq k\leq \lfloor{n/2}\rfloor$,
there is an explicit bijection $\psi: \mathcal{AP}_{n, n-2k}\to \mathcal{Y}_{n,k}\times \mathcal{DP}_{k}$ such that 
if   $\psi(u)=(y,p)$ with for $u\in \mathcal{AP}_{n, n-2k}$ and  $(y,p)\in \mathcal{Y}_{n,k}\times  \mathcal{DP}_{k}$ then
$w(u)=w(p)$, where the weight is associated to the sequences $(b_i)$ and $(\lambda_i)$
with $b_{2i}=1$, $b_{2i+1}=0$  ($i\geq 0$),  
and $\lambda_{k}=\lfloor \frac{k+1}{2}\rfloor {t}$ ($i\geq 1$).
\end{lem}
\begin{proof} 
Since an Andr\'e  path (word)  on $\{\mathsf{U, D, L}\}$ has only level-steps 
at even height and starts from height 0, so the subword 
  between two consecutive level-steps $\mathsf{L}$'s must be of even length and is  a word on the alphabet  $\{\mathsf{UU}, \mathsf{DD}, \mathsf{UD}, \mathsf{DU}\}$.
Thus, 
any Andr\'e word $u\in \mathcal{AP}_{n, n-2k}$  can be  written  
in a unique way as follows:
$$
u=\mathsf{L}^{y_1}w_1\mathsf{L}^{y_2}w_2\ldots w_k\mathsf{L}^{y_{k+1}}\quad\textrm{with}\quad  w_i\in \{\mathsf{UU, DD, UD, DU}\}.
$$ 
Let  
$y:=(y_1, \ldots, y_{k+1})$ and
$p:=w_1\ldots w_k$. As the  path $p$ is obtained by removing out all  the level-steps $\mathsf{L}$'s from  the Andr\'e path $u$, each step in $p$ keeps 
the same height in $u$, and $(y,p)\in \mathcal{Y}_{n,k}\times  \mathcal{DP}_{k}$,
Let $\psi(u)=(y, p)$. It is clear that this is the desired bijection.
\end{proof}

\begin{thm}\label{formula(-1)} For $n\geq 1$ we have 
\begin{align}\label{t-Han}
D_{n}(1,-1,t)=\sum_{k=0}^{n-1}\binom{n-1-k}{ k}k!t^k.
\end{align}
\end{thm}
\begin{rem}  
If $t=1$, in view of \eqref{q-ES-D}, the above  result is equivalent  to
\begin{align}\label{Han-eq}
E_n(-1)=\sum_{k=0}^{n-1}\binom{n-1-k}{ k}k!,
\end{align}
which was posted  by P. Barry, see A122852 in OEIS~\cite{OEIS}.
Han~\cite[Theorem~7.1]{Han19} gave a non-trivial (sic) proof of  \eqref{Han-eq} by showing that 
both sides of \eqref{Han-eq}  satisfy the same recurrence relation, which had been conjectured by 
R. J. Mathar.  Our proof of \eqref{t-Han} is combinatorial and  insightful for the summation formula in \eqref{t-Han}.
\end{rem}

\begin{proof}[Proof of Theorem~\ref{formula(-1)}]
By Lemmas~\ref{lem1} and \ref{lem2} we have 
$$
D_{n+1}(1,-1, t)=\sum_{k\geq 0}\sum_{(y,p)\in \mathcal{Y}_{n,k}\times \mathcal{DP}_{k}}w(p).
$$
Since the cardinality of $\mathcal{Y}_{n,k}$  is $\binom{n-k}{k}$, and 
the generating polynomial of $\mathcal{DP}_k$ is equal to $k! t^k$ by Lemma~\ref{Euler-Flajolet}, summing over all
 $0\leq k\leq \lfloor{n/2}\rfloor$ we obtain Theorem~\ref{t-Han}.
\end{proof}

\begin{rem}
It is a challenge to   show directly that   Theorem~\ref{formula(-1)} is the limit case 
of Theorem~\ref{formula(p=1)}  when $q\to -1$.
\end{rem}


\end{document}